\documentclass[10pt]{article}

\usepackage{amssymb,latexsym,amsmath,amsthm}
\usepackage{graphicx}
\usepackage[shortlabels]{enumitem}
\usepackage{subcaption}
\usepackage{thm-restate}
\usepackage[numbers,square,comma]{natbib}
\usepackage[hidelinks]{hyperref}

\newtheorem{theorem}{Theorem}
\newtheorem{corollary}{Corollary}
\newtheorem{lemma}{Lemma}

\newtheorem{observation}{Observation}

\theoremstyle{definition}
\newtheorem{definition}{Definition}

\newtheorem{remark}{Remark}
\newtheorem{example}{Example}
\newtheorem{question}{Question}

\title{Universality of Infinite Chess}
\author{
Matthew Bolan%
\thanks{\begin{tabular}[t]{@{}l@{}}
Department of Mathematics, University of Toronto, Canada\\
\texttt{matthew.bolan@mail.utoronto.ca}
\end{tabular}}
\and
Andreas Tsevas%
\thanks{\begin{tabular}[t]{@{}l@{}}
Department of Physics, Ulm University, Germany\\
\texttt{andreas.tsevas@uni-ulm.de}
\end{tabular}}
}
\date{}

\begin{document}

\maketitle

\begin{abstract}
We prove that chess played on the infinite chessboard $\mathbb{Z}^2$ with infinitely many pieces is as powerful as it could possibly be, by showing that every open Gale-Stewart game with draws is strategically equivalent to some infinite chess position and vice versa. As our construction is computable and open Gale-Stewart games are well understood, this allows us to resolve many open questions about the complexity of infinite chess with infinitely many pieces. In particular, all countable ordinals arise as the game value of some such chess position. We also give an alternate construction that realizes all countable ordinals as game values, with the pleasing property that it consists only of the king pair and pawns.
\end{abstract}


\vspace{0.5em}

\section{Introduction}
\label{sec:intro}
    This article resolves several open questions regarding the game values and the complexity of infinite chess played with infinitely many pieces. Infinite chess is a natural extension of the rules of chess to allow play on the integer lattice $\mathbb Z^2$ \cite{brumvele2012mate, evans2013transfinite, evans2017position}. This extension is mostly standard: Queens, bishops, and rooks move in the same directions as in the finite case, but may now move any finite number of squares, since they are no longer confined by the edges of the chessboard. Kings, knights, and pawns move as expected, though there is no concept of pawn promotion or castling, and pawns are assumed never to be able to move two squares at once. Each player has exactly one king, infinite play is defined to be a draw, and the rules of check and checkmate, stalemate, threefold repetition, and the fifty-move rule\footnote{In the previous literature on infinite chess, draws by the fifty-move rule and threefold repetition are usually omitted, but we sidestep the question of whether to include them by producing positions whose outcome is not affected by these rules.} remain unchanged.

An infinite chess \emph{position} is given by an indicator of whose turn it is together with an assignment of a value to each square on the board indicating whether there is a piece there and, if so, which piece.

Infinite chess in its current form was popularized through questions on MathOverflow \cite{MO2011Checkmate, MO2010Decidable}. As was first discussed in the answers to \cite{MO2011Checkmate} and shown in \cite{evans2013transfinite}, there exist positions where White can always force mate in finite time, but Black can choose to delay the mate by any unbounded number of moves. This leads to the consideration of transfinite ordinal \emph{game values} (Definition~\ref{def:game_values}) which measure the distance White is from winning.

Although researchers have been able to construct positions with a game value of $\omega$ and all the way up to $\omega^4$ \cite{evans2017position}, it has been a longstanding open question how much higher these game values can go in infinite chess. In particular, Evans and Hamkins asked in
\cite{evans2013transfinite} what the supremum of the game values of all the winning positions for White is in infinite chess. This ordinal is referred to as \emph{the omega one of infinite chess} and is denoted $\omega_1^{\mathfrak{Ch}_{\hskip-2ex{\genfrac{}{}{0pt}{}{ }{\sim}}}}$. The only known upper bound for $\omega_1^{\mathfrak{Ch}_{\hskip-2ex{\genfrac{}{}{0pt}{}{ }{\sim}}}}$ is the observation that all game values of chess are countable i.e., $\omega_1^{\mathfrak{Ch}_{\hskip-2ex{\genfrac{}{}{0pt}{}{ }{\sim}}}} \leq \omega_1$, which follows from the fact that the game tree of infinite chess is only countably branching.

In this paper, we present two distinct constructions, both of which resolve this question and go beyond it. The first is presented in Section~\ref{sec:bishop-tree}. The proof strategy is to embed certain games on trees into the chessboard, similarly to the strategy used in \cite{evans2013transfinite} to show all countable game values arise in three-dimensional infinite chess. The primary obstacle we overcame in the two-dimensional case was how to allow branches of the tree to cross each other. This was achieved by representing the branches of our trees not as structures made of chess pieces but as empty diagonals along which bishops travel. The nature of our trees also allows us to treat the two players essentially identically, allowing us to build arbitrary game trees and so simulate all open Gale-Stewart games with draws (Definition~\ref{def:GSGame}).

The statement and various consequences of our results are defined in terms of descriptive set theory --- we refer the reader to Section IV of \cite{Piergiorgio1989} for any undefined terms and notations. This is a matter of formalization only --- a reader merely interested in the chess constructions should be able to follow and appreciate them without any background in set theory. Our first main theorem is as follows:
\begin{restatable}{thm}{mainthm}
    \label{thm:main}
    For any open Gale-Stewart game with draws $G_{S_1, S_2}$, there exists an infinite chess position $p$ having the following properties:
    \begin{enumerate}[(i)]

        \item The position $p$ can be computed given oracles for $S_1$ and $S_2$, and so is in $\Delta_1^{0,(S_1, S_2)}$.

        \item White (resp.\ Black) can force a win from $p$ if and only if player one (resp.\ two) can force a win in $G_{S_1, S_2}$. If $G_{S_1, S_2}$ has game value $\alpha$ for some player, then $p$ has game value $\beta \ge \alpha$ for the corresponding player.

        \item For any winning (resp.\ drawing) strategy $\sigma$ for some player in $p$, there exists a $\Delta_1^{0,\sigma}$ winning (resp.\ drawing) strategy for the corresponding player in $G_{S_1, S_2}$. For any winning (resp.\ drawing) strategy $\sigma$ for some player in $G_{S_1, S_2}$, there exists a $\Delta_1^{0,\sigma}$ winning (resp.\ drawing) strategy for the corresponding player in $p$.
        
    \end{enumerate}
\end{restatable}

As infinite chess itself turns out to be such a Gale-Stewart game with draws, infinite chess with infinitely many pieces is as powerful and as strategically complex as it could possibly be. In particular, we immediately obtain the following corollary, resolving the question of the omega one of infinite chess:
\begin{corollary}
    \label{cor:game-value}
    For every countable ordinal $\alpha<\omega_1$, there exists an infinite chess position with infinitely many pieces with game value $\alpha$.
\end{corollary}

We remark that by the construction's computable nature, the same argument restricted to games with $S_1$ and $S_2$ computable shows that the omega one of computable chess positions is the Church-Kleene ordinal $\omega_1^{\text{CK}}$, where the upper bound was proven in \cite{evans2013transfinite}.

Further, many classical results about the complexity of Gale-Stewart games apply to infinite chess. In particular, Theorem 2 from Andreas Blass' 1972 paper `Complexity of winning strategies' \cite{blass1972strategies} immediately yields the following corollary to our main theorem:

\begin{corollary}
    For any hyperarithmetic set $A$ there exists a computable infinite chess position $p$ such that $A$ is recursive in every winning strategy for the associated chess game.
\end{corollary}

Additionally, by the construction in the proof of Theorem 3 of the aforementioned paper, we obtain the following corollary:
\begin{corollary}
\label{cor:hyper}
    There exists an infinite chess position $p$ with the following properties:
    \begin{enumerate}[(i)]
        \item The position $p$ is computable.
        \item White has a computable winning strategy defeating every hyperarithmetic strategy for Black.
        \item Allowing arbitrary play, neither player has a winning strategy, and both players can force a draw.
    \end{enumerate}
\end{corollary}

These results pertain to the standard formulation of infinite chess in which queens, rooks, and bishops can move arbitrarily large distances. However, there is also a variant of the game in which all pieces are restricted to a maximum movement distance of seven squares per move, matching the largest move possible on the standard $8\times 8$ board. This version is likely the first to appear in the literature, having been considered by Dénes Kőnig in his 1927 paper about his eponymous lemma \cite{konig}.

Our second construction, presented in Section~\ref{sec:pawn-tree}, addresses this variant. The compactness argument used to prove our Theorem~\ref{thm:short-range-zugzwang} shows that transfinite game values under this strong restriction can only arise via the zugzwang\footnote{Zugzwang is the German word for `move compulsion'. In chess parlance it denotes the phenomenon of a player being put at a disadvantage due to their obligation to make a move.} phenomenon of chess, which makes the simulation of arbitrary games difficult. Nonetheless, we are able to obtain the following result:

\begin{restatable}{thm2}{secondthm}
\label{thm:game-values-pawns}
For every countable ordinal $\alpha<\omega_1$, there exists an infinite chess position with game value $\alpha$ consisting only of the king pair and pawns.
\end{restatable}

This is proven by simulating a small class of games (Definition~\ref{def:choosing_from_z_game}) which take full advantage of zugzwang, attain arbitrarily large countable game values, and can be realized on the chessboard as a binary tree made of pawn chains.

\section*{Acknowledgments}
We wish to thank the founders and maintainers of MathOverflow for enabling the initial discussions of infinite chess which led to this work. We also wish to thank Joel David Hamkins and Davide Leonessi for their organization of the Infinite Games Workshop in 2023, where we were able to present and discuss preliminary results on infinite chess. We are also grateful to Hunter Spink for his comments on an early version of this manuscript. 

Matthew Bolan acknowledges and is grateful for the support of an Ontario Graduate Scholarship.

\section{Simulating Arbitrary Games in Infinite Chess}
\label{sec:bishop-tree}
In this section we describe our first construction, proving Theorem~\ref{thm:main}. Section~\ref{sec:open-games-with-draws} is devoted to various technical definitions related to the type of games we simulate in infinite chess, while Section~\ref{sec:bishop-tree-chess} deals with the construction of the actual positions. The reader only interested in the chess can skip to Section~\ref{sec:bishop-tree-chess} and still get a strong idea of the power of the construction.

\subsection{Open Gale-Stewart Games with Draws}
\label{sec:open-games-with-draws}

We first introduce the class of games that we wish to simulate. For a set $S \subseteq \mathbb N ^{< \omega}$ of finite sequences of natural numbers, we define the set $U(S) \subseteq \mathbb N ^{\omega}$ to be the set of all infinite sequences of naturals having some initial segment belonging to $S$.

\begin{definition}[\emph{Open Gale-Stewart game with draws}]
\label{def:GSGame}
Let $S_1,S_2\subseteq\mathbb N^{< \omega}$ be such that $U(S_1)\cap U(S_2) = \varnothing$. We define the \emph{open Gale-Stewart game with draws} $G_{S_1, S_2}$ as follows: Two players take turns naming elements of $\mathbb N$, proceeding endlessly so that a sequence $s \in \mathbb{N}^\omega$ is produced. We say the first player wins $G_{S_1, S_2}$ if $s\in U(S_1)$, the second player wins if $s\in U(S_2)$, and otherwise the game is drawn. The sets $U(S_1)$ and $U(S_2)$ are called the payoff sets.
\end{definition}

Definition~\ref{def:GSGame} is a generalization of open Gale-Stewart games, which were first defined in \cite{GaleStewart1953}, to allow for three possible outcomes: wins, losses, and draws. The game $G_{S_1, S_2}$ is a Gale-Stewart game in the classical sense precisely when $U(S_1) \cup U(S_2) = {\mathbb N}^{\omega}$, meaning no sequence allows for a draw. We define the payoff sets in terms of $S_1$ and $S_2$ instead of using $U(S_1)$ and $U(S_2)$ directly so that statements about our construction, game values, and computability results are easier to state.

The reasoning for the terminology \emph{open} is that $U(S_1)$ and $U(S_2)$ are open in the product topology on $\mathbb{N}^\omega$, and all open sets in the product topology are of the form $U(S)$ for some $S \subseteq \mathbb N^{< \omega}$. In the literature, more general Gale-Stewart games are considered which permit payoff sets that are not open. If a player in an open Gale-Stewart game with draws has won, then by the definition of the product topology there was some finite time at which the win was guaranteed regardless of how either player chose to continue. With our particular definition, one can check if some player is guaranteed to win by checking if any initial segment of the named sequence belongs to $S_1$ or $S_2$.

A \emph{strategy} in a Gale-Stewart game with draws is a map $\sigma: \mathbb N^{< \omega} \to \mathbb N$ from finite sequences of naturals to the naturals. A player follows the strategy $\sigma$ by playing $\sigma(v_1,v_2,\ldots,v_n)$ if it is their turn after the finite sequence of naturals $(v_1,v_2,\ldots,v_n)$ has been played\footnote{Note that this definition assigns a move to every finite sequence, even those which can never arise and those at which the player is not to move. These values are immaterial and play no role in the execution of the strategy.}. Strategies are winning (resp.\ drawing) if they win (resp.\ draw or win) against all other strategies.

There is a concept coming from infinitary game theory, which assigns to open games an ordinal measuring the winning player's time until victory. We define this now for open Gale-Stewart games with draws, following \cite{evans2013transfinite}.
\begin{definition}[\emph{Ordinal game values}] \label{def:game_values}
    Consider an open Gale-Stewart game with draws $G_{S_1, S_2}$. Given some initial play in that game specified by a finite sequence of naturals $v = (v_1,\ldots,v_n)$, we define the \emph{ordinal game value of $v$ for player one} inductively as follows: If player one has already guaranteed his victory in the game i.e., if there exists some natural number $k\leq n$ such that $(v_1,\ldots,v_k)\in S_1$, then $v$ has game value $0$. If it is the turn of player two, we say $v$ has game value $\alpha$ if for every $m \in\mathbb N$ the play $(v_1,\ldots,v_n,m)$ has already been assigned game value $\alpha_m \le \alpha$ and $\alpha = \sup_{m \in \mathbb N} \{\alpha_m\}$. If it is the turn of player one we say $v$ has game value $\alpha$ if the set $T_{v, \alpha}$ of naturals $m$ such that the play $(v_1,\ldots,v_n,m)$ has already been assigned game value $\alpha_m < \alpha$ is non-empty and $\alpha = \min_{m \in T_{v, \alpha}} \{ \alpha_m \} + 1$. The ordinal game values for player two are defined analogously, with the roles of the players and payoff sets interchanged.
\end{definition}
The fundamental observation about game values of open games is that the plays having an ordinal game value for player one are precisely the plays that are winning for player one. In fact, player one has a winning strategy by always choosing some move leading to a smaller game value.

Any given infinite chess position can be interpreted as an open Gale-Stewart game with draws. To do this, we fix some encoding of chess moves as natural numbers, declare illegal moves to be an instant loss, and allow any natural number as a move after the corresponding chess game would have ended. The resulting game is open, as checkmate can only occur after finite play and infinitely long play on the chessboard is a draw. The set $S_1$ can then be taken to be those finite sequences of play which result in Black being checkmated or Black having made an illegal move, and $S_2$ the analogous set for White. In this way, infinite chess positions can be assigned game values. If a position has a finite game value $n$ for White, then this notion agrees with the notion of White having mate-in-$n$ common to chess puzzles \cite{evans2013transfinite}.

Our main theorem is that infinite chess can simulate any open Gale-Stewart game with draws:

\mainthm*

The position $p$ will essentially be a simple embedding of the important parts of the game tree of $G_{S_1, S_2}$ on the chessboard. Moves in $G_{S_1, S_2}$ will correspond in a computable way to certain bishop moves in $p$, at least until checkmate occurs.

Our construction must address the fact that Gale-Stewart games take infinitely long to play, but infinite play in chess is defined as a draw. Thus, any chess position simulating an open Gale-Stewart game must at some point allow a player who has guaranteed himself victory in the corresponding Gale-Stewart game to checkmate the opponent. We introduce the following definitions to pinpoint which moves by the losing player in a Gale-Stewart game do not further affect its outcome.
\begin{definition}[\emph{Nodes and losing nodes}] \label{def:losing-nodes}
    Consider an open Gale-Stewart game with draws $G_{S_1, S_2}$ and some initial play in that game specified by a finite sequence of naturals $v = (v_1,\ldots,v_n) \in {\mathbb N}^n$ for $n\in\mathbb N$. The play $v$ is called a \emph{node} of $G_{S_1, S_2}$. Further, $v$ is called \emph{losing for player one} if it is the turn of player one (i.e., $n$ is even) and some initial segment of $v$ belongs to $S_2$. Likewise, $v$ is called \emph{losing for player two} if it is the turn of player two (i.e., $n$ is odd) and some initial segment of $v$ belongs to $S_1$.
\end{definition}

\begin{definition}[\emph{Reduced game tree}] \label{def:reduced-game-tree}
    Consider an open Gale-Stewart game with draws $G_{S_1, S_2}$. Then its \emph{reduced game tree} $T \subseteq \mathbb N^{<\omega}$ consists of precisely those nodes $v\in\mathbb N^{<\omega}$ such that $v$ has no proper initial segment that is a losing node for either player in the game.
\end{definition}

The reduced game tree $T$ of an open Gale-Stewart game with draws can be viewed as the tree containing precisely those moves which are necessary to play the game to its conclusion such that terminal nodes always correspond to losing for the player to move. We record some properties of the reduced game tree that immediately follow from the definitions above:

\begin{observation} \label{obs:reduced-game-tree-prop}
    Given an open Gale-Stewart game with draws $G_{S_1, S_2}$ and its reduced game tree $T$, the following hold:

    \begin{enumerate}[(i)]
        \item The tree $T$ can be computed given oracles for $S_1$ and $S_2$, and so is in $\Delta_1^{0,(S_1, S_2)}$.
        
        \item If $v\in {\mathbb N}^{<\omega}$ is a terminal node of $T$, then $v$ is a losing node of $G_{S_1, S_2}$. In particular, any continuation of the game after $v$ results in a loss for the player whose turn it was after the initial play $v$.
        
        \item If $v\in {\mathbb N}^{\omega}$ is an infinite branch of $T$, then $v$ is a draw in $G_{S_1, S_2}$.
    \end{enumerate}
\end{observation}

Thus, an open Gale-Stewart game with draws can be recast as the players alternately making moves on its reduced game tree $T$ by moving deeper into the tree, with both of them aiming to force their opponent onto a terminal node.

\subsection{Infinite Chess Construction}
\label{sec:bishop-tree-chess}

We now explicitly outline the embedding of arbitrary open Gale-Stewart games with draws $G_{S_1, S_2}$ into the infinite chessboard. To do this, we build a chess position corresponding to its reduced game tree $T$. To that end, we introduce several small chess components along the way, which can be assembled to represent the nodes of the reduced game tree $T$ on the chessboard.

\begin{figure}[!ht]
    \centering
    \includegraphics[width=0.6\linewidth]{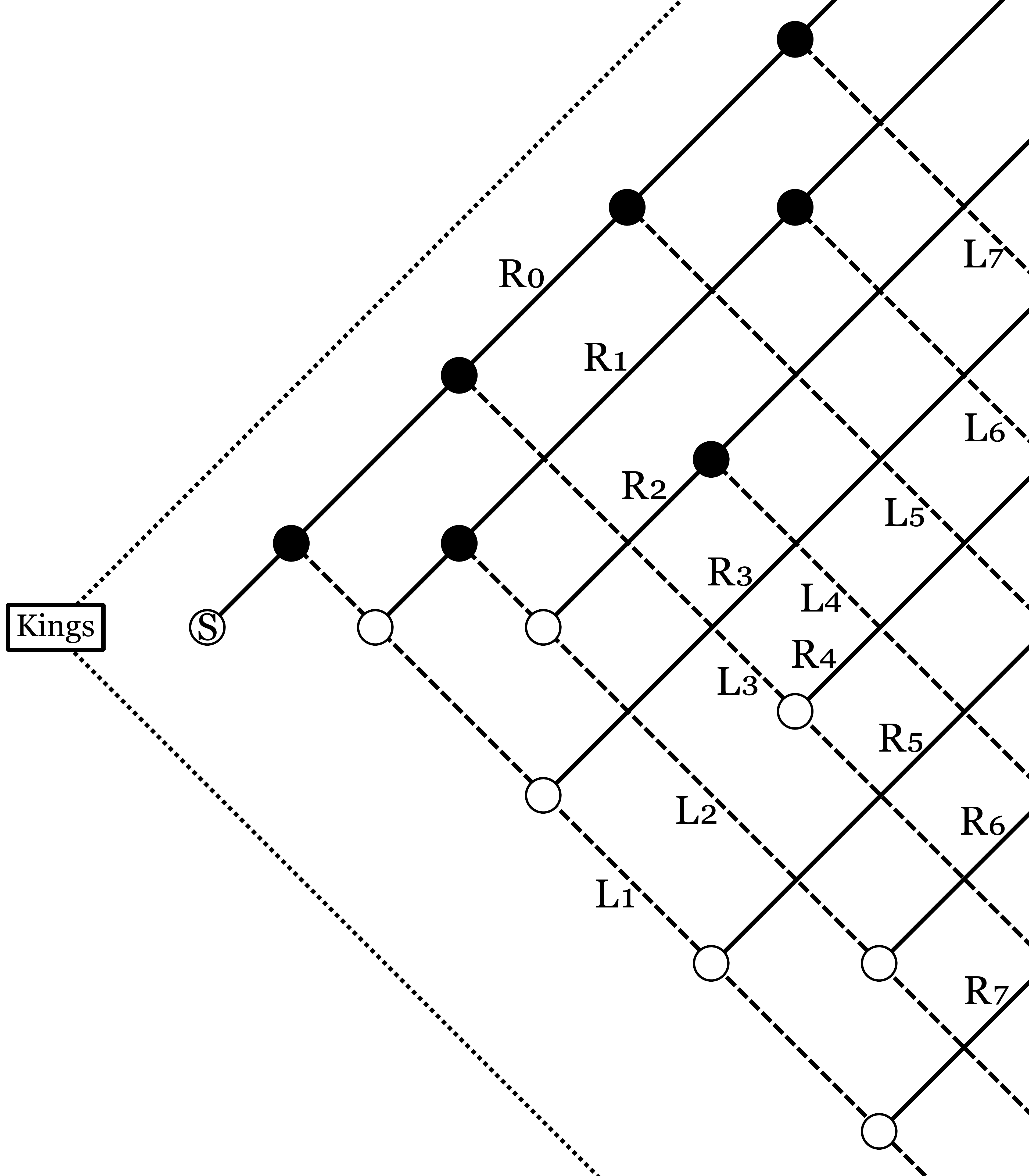}
    \caption{Global structure of an infinite chess position corresponding to an open Gale-Stewart game with draws. White circles are white nodes and black circles are black nodes. The leftmost white circle marked S is the root node of the tree. Diagonals indicate possible bishop paths through the tree, which are allowed to pass through nodes.}
    \label{fig:bishop-tree-structure}
\end{figure}

Before introducing any chess components, we give a high level overview of the construction. Figure~\ref{fig:bishop-tree-structure} illustrates the global structure of the chess position corresponding to some open Gale-Stewart game with draws $G_{S_1, S_2}$. The black and white kings are trapped on the very left of the position where they can be checkmated by bishop attacks along the dotted checkmate diagonals. To the right of them, a bi-colored tree corresponding to the reduced game tree $T$ of the game is embedded on the chessboard. 

Both players will be extremely limited in their movement options and will each only control a single free bishop at a time. The bishops are chasing each other through the tree with the goal of being the first to escape it and checkmate their opponent's vulnerable king. The bishops start off in the leftmost node, the white root node of the tree marked with an `S'. The position we will construct forces White to move their bishop towards the top right along the solid line diagonal $R_0$ to one of the potentially infinitely many black nodes along that line, which shall correspond to player one's first move in the Gale-Stewart game. Black will be forced to respond by following and capturing the white bishop, which kicks off a series of local moves in the black node, ending with Black being forced to leave the node toward the bottom right along a dashed diagonal to a white node of his choice. We shall see that this will correspond to player two's first move in the Gale-Stewart game. At the newly reached white node, local play will ensue and White will be the next player to make a choice, this time toward the top right along the corresponding solid diagonal to a black node of his choice. This alternation between white choices along solid diagonals to the top right and black choices along dashed diagonals to the bottom right continues, perhaps indefinitely, luring the players deeper and deeper into the tree. In general, the game will end if and only if a player finds himself at a node with no children. That player loses the game by not being able to choose a successor node of the opposite color, which will correspond to having reached a terminal node of the reduced game tree of the Gale-Stewart game.

To describe the specific positions of all nodes of a full bi-colored tree $\mathbb N^{< \omega}$ on the chessboard, we let $k$ be a parameter\footnote{In practice, we will choose $k=50$, as shown in Figure~\ref{fig:infinite-game-tree} later on.} representing the spacing between adjacent parallel diagonals. For $n \in \mathbb Z$, let $R_n = \{(x,y)|x-y=kn\}$ be the $n$th Northeast facing diagonal, and likewise $L_n = \{(x,y)|x+y=kn\}$ the $n$th Southeast facing diagonal. We build the start node on the intersection of $R_0$ and $L_0$, and the king chambers on the intersection of $R_{-1}$ and $L_{-1}$. If $n \ge 0$, $m \ge 0$, and $m \equiv 2^n \pmod {2^{n+1}}$ then we build a black node on the intersection of $R_n$ and $L_m$, and likewise if $n > 0$, $m \ge 0$, and $m \equiv 2^{n-1} \pmod {2^{n}}$ then we build a white node on the intersection of $R_m$ and $L_n$. We now record some facts which we will need in our arguments for the correctness of the local components:

\begin{observation}
\label{obs:global-properties}
For $k$ large enough, the layout described above and depicted in Figure~\ref{fig:bishop-tree-structure} fits on the chessboard and enjoys the following properties:
    \begin{enumerate}[(i)] 
        \item For each $n \ge 0$ there is exactly one white or start node on each $R_n$. Similarly, for each $n > 0$ there is exactly one black node on each $L_n$.
        \item For any $n\ge 0$ the black nodes on $R_n$ are each to the Northeast of the unique white or start node on $R_n$. Likewise, for $n \ge 1$ the white nodes on $L_n$ are to the Southeast of the unique black node on $L_n$. \label{obs:global-properties-clear-paths-behind} 
        \item For large enough $k$, the only pieces within $20$ squares of a diagonal labeled $R_n$ or $L_n$ belong to components explicitly placed on these diagonals by the construction. \label{obs:global-properties-wiggle-room}
\end{enumerate}
\end{observation}

We now construct the king chamber and the various nodes needed to realize this structure on the chessboard. In all figures in this paper one may assume that the white pawns move upward (north) and the black pawns move downward (south). When analyzing the various components we must consider cases where one player passes locally by playing some move outside the figure. We tacitly assume that such moves do not end the game prematurely, lead to infinitely long forced sequences, or cause unexpected pieces to enter the relevant figure. These assumptions will be justified upon assembly of the final position.

\begin{figure}[!ht]
    \centering
    \begin{subfigure}{0.4\textwidth}
        \centering
        \includegraphics[scale=0.8]{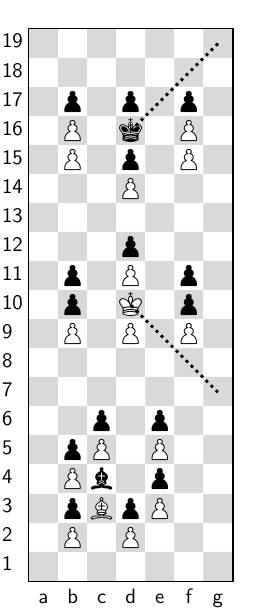}
        \caption{King chambers}
        \label{fig:King-Chambers}
    \end{subfigure}
    \begin{subfigure}{0.4\textwidth}
        \centering
        \includegraphics[scale=0.8]{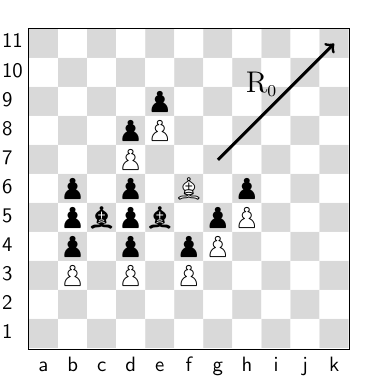}
        \caption[]{\tabular[t]{@{}l@{}}White root node \\ (White to move)\endtabular}
        \label{fig:Start-Node}
    \end{subfigure}
    \caption{Tree Root Components}
\end{figure}

First we introduce the king chambers. 

\begin{lemma} \label{lem:king-chamber-1}
    The king chamber chess component, shown in Figure~\ref{fig:King-Chambers}, obeys the following properties.
    \begin{enumerate}[(i)]
        \item Bishop checks from the dotted diagonals that cannot be captured or interfered with are checkmate.
        \label{lem:king-chamber-1i}
        \item Neither player can be stalemated provided no piece has entered the region depicted in Figure~\ref{fig:King-Chambers}.
        \label{lem:king-chamber-1ii}
    \end{enumerate}
\end{lemma}
\begin{proof}
    Property~\ref{lem:king-chamber-1i} holds as the only way to escape such a check is to move out of it, but neither king currently has any moves. Property~\ref{lem:king-chamber-1ii} holds as both players will have an infinite source of moves with the trapped stalemate prevention bishops on {\tt c3} and {\tt c4} in this case.
\end{proof} 

At the beginning of the game with White to move, the only pieces besides the stalemate prevention bishops that have available legal moves will be found in the \emph{root node}, which is shown in Figure~\ref{fig:Start-Node}. The root node will lie somewhere to the right of the king chambers as we saw in Figure~\ref{fig:bishop-tree-structure}, where it is the leftmost white node marked with S.

\begin{lemma} \label{lem:start-node}
    Suppose the bishops depicted in Figures~\ref{fig:King-Chambers} and~\ref{fig:Start-Node} are the only pieces with moves available on the entire chessboard, Black guards no squares along the diagonal ray indicated by the solid arrow, and White is to move. 
    \begin{enumerate}[(i)]
        \item If White does not move their bishop in the start node (Figure~\ref{fig:Start-Node}) to a protected square in the direction indicated by the solid arrow, Black can ensure that a black bishop is able to leave Figure~\ref{fig:Start-Node} and that White's free pieces consist of at most a pawn in the start node. \label{lem:start-node-no-white-deviations}
        \item If White moves their bishop in the direction indicated by the solid arrow to a square where it threatens to mate Black next turn, Black must immediately capture the bishop with his bishop on {\tt e5} lest he be checkmated in finitely many moves. The only pieces remaining in the start node after such a capture are pawns which cannot move. \label{lem:start-node-no-black-deviations}
    \end{enumerate}
   
\end{lemma}
\begin{proof}
    For~\ref{lem:start-node-no-white-deviations}, observe that any of the five moves (including a move by a stalemate prevention bishop) not in the direction of the arrow allow for the immediate capture of White's bishop in a way which frees one of Black's bishops on {\tt c5} and {\tt e5} in the start node, while White frees at most a single pawn. Similarly, a move to an unprotected square in the direction of the arrow allows Black the safe capture of White's bishop in a manner which leaves Black with a free bishop.
    
    For~\ref{lem:start-node-no-black-deviations}, observe that if Black fails to capture then White can usually give mate immediately. In the case where Black's bishop threatens to interfere with the check, White may capture Black's necessarily undefended bishop, return to the square where the threat was originally made, and checkmate on the turn after that.
\end{proof}

\begin{figure}[!ht]
    \centering
    \begin{subfigure}{0.49\textwidth}
        \centering
        \includegraphics[width=1.13\linewidth]{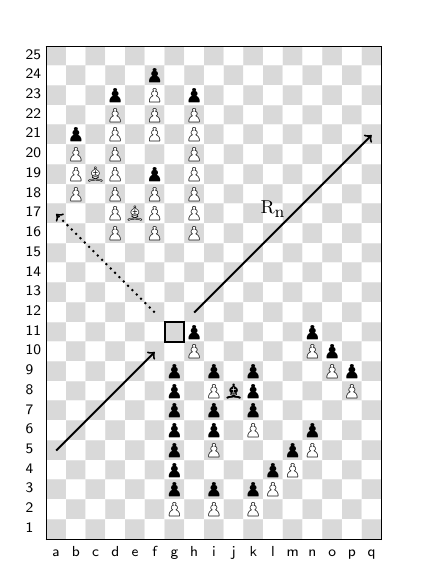}
        \caption{Pristine black node}
        \label{fig:Black-Node-pristine}
    \end{subfigure}
    \begin{subfigure}{0.49\textwidth}
        \centering
        \includegraphics[width=1.13\linewidth]{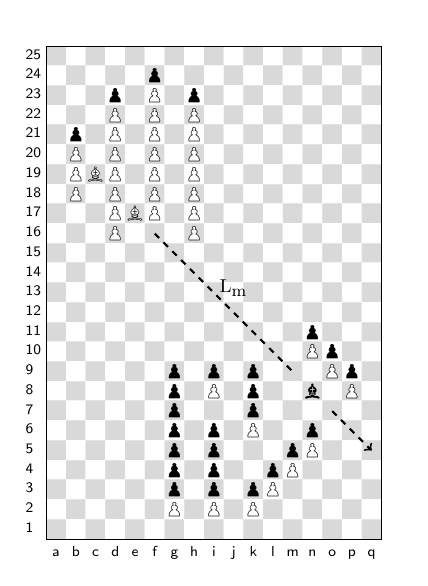}
        \caption{Used black node (Black to move)}
        \label{fig:Black-Node-used}
    \end{subfigure}
    \caption{Black Node. The sequence {\tt 1.\ Bg11 Bxg11 2.\ hxg11 h10 3.\ g12 h9 4.\ g13 h8 5.\ g14 h7 6.\ g15 h6 7.\ g16 hxi5 8.\ g17 i4 9.\ g18 i5 10.\ gxf19 i6 11.\ f20 Bi7 12.\ f19 Bk5 13.\ f18 Bn8 14.\ f17} in (a) leads to the position in (b).}
    \label{fig:Black-Node}
\end{figure}

The only protected squares along the diagonal that White can move to are contained in so-called black nodes, indicated by black circles in Figure~\ref{fig:bishop-tree-structure}. An example of a black node is shown in Figure~\ref{fig:Black-Node-pristine}. Lemma~\ref{lem:start-node} implies the white bishop exiting the start node will be forced to move to the square labeled {\tt g11} in Figure~\ref{fig:Black-Node-pristine} in order to prevent capture without compensation. As the solid diagonal is unobstructed, White may move to any such node along the diagonal, corresponding to White's choice of first move in the corresponding Gale-Stewart game. Crucially, after the move to {\tt g11} it is White who threatens to win! Due to the global arrangement of the nodes (as depicted in Figure~\ref{fig:bishop-tree-structure}), there is nothing blocking the dotted diagonal to the top left of any black node. As the king chamber is somewhere on the left of the black node, if given a free move White can give mate immediately by moving his bishop along the dotted diagonal.

Thus, in order to avoid being immediately mated, Black is forced to recapture on {\tt g11} with his single free bishop. 

\begin{lemma}
    Suppose White is to move and the only free piece on the board is a Black bishop on the highlighted square {\tt g11} in Figure~\ref{fig:Black-Node-pristine}.

    \begin{enumerate}[(i)]
        \item White can play to ensure that either the position in Figure~\ref{fig:Black-Node-used} occurs, or a white bishop is able to escape the bounds of the figure along the dotted ray before any black bishop can escape the bounds of the figure or guard any square on the dotted ray. 
        \item Black can play to ensure that either the position in Figure~\ref{fig:Black-Node-used} occurs, or a black bishop escapes the figure along the ray $L_m$ two moves before any white bishop can escape the bounds of the figure. 
    \end{enumerate}
    
\end{lemma}

\begin{proof}
Consider the sequence of moves {\tt 2.\ hxg11 h10 3.\ g12 h9 4.\ g13 h8 5.\ g14 h7 6.\ g15 h6 7.\ g16 hxi5 8.\ g17 i4 9.\ g18 i5 10.\ gxf19 i6 11.\ f20 Bi7 12.\ f19 Bk5 13.\ f18 Bn8 14.\ f17}, the result of which is shown in Figure~\ref{fig:Black-Node-used}. We argue that neither side can deviate from this sequence without allowing their opponent's bishop to escape in the needed way. Most of the moves in the sequence are the only legal move besides waiting moves with the stalemate prevention bishop. The moves {\tt 4...hxi8, 7...h5} and {\tt 10. g19} are all unthinkable, for they each permanently destroy the possibility of freeing one's own bishop while still permitting the opponent to free theirs. The fact that {\tt 7...h5} may allow some black pawns to exit the figure is not a concern given that there are no pieces immediately south of the node by Observation~\ref{obs:global-properties}\ref{obs:global-properties-wiggle-room}.

If White plays a waiting move with the stalemate prevention bishop during that sequence, then the dashed diagonal in Figure~\ref{fig:Black-Node-used} will be unguarded on move 14, allowing Black to escape with the bishop along this diagonal before White can play {\tt f17}. Similarly, if Black plays a waiting move during this sequence, White will play {\tt f17} before Black can move his bishop onto $L_m$, letting White maneuver his bishop to {\tt g11} without worry. The remaining deviations are all similar to playing waiting moves. If White deviates with one of {\tt 4.\ ixh9, 7.\ ixh6, 13.\ Bf18,} or {\tt 13.\ Bg19}, he will lose the race by at least two tempi, or as if he played two waiting moves. If Black deviates with {\tt 9...fxg18}, White can play {\tt 10.\ fxg18} and {\tt 11.\ f17} by force, winning the race handily. As this is an exhaustive list of all deviations by either player, the claimed sequence is indeed forced.
\end{proof}

The final position in that node after the moves above is now shown in Figure~\ref{fig:Black-Node-used}. The black bishop is now standing unguarded on {\tt n8} on the release diagonal, which is marked with a dashed line. Meanwhile, the white bishop on {\tt e17} threatens its capture along this diagonal. All other pieces in the node are completely deadlocked and neither player can make a move in this component except with those newly liberated bishops. Black to move should not capture the white bishop with {\tt ...Bxe17}, since he loses his only free piece and the response of {\tt dxe17} will lead to the subsequent release of the white bishop on {\tt c19} and checkmate soon after. Black should also not leave the dashed release diagonal, lest he wishes to allow the newly liberated white bishop to checkmate his king as if he had played a waiting move. Thus, Black's only option is to move along the dashed diagonal toward the Southeast to some guarded square. Crucially, Black may choose among the potentially infinitely many such squares. This freedom by Black to choose a destination square with his bishop starting from this node motivates the name `black node'.

\begin{figure}[!ht]
    \centering
    \begin{subfigure}{0.49\textwidth}
        \centering
        \includegraphics[width=1.13\linewidth]{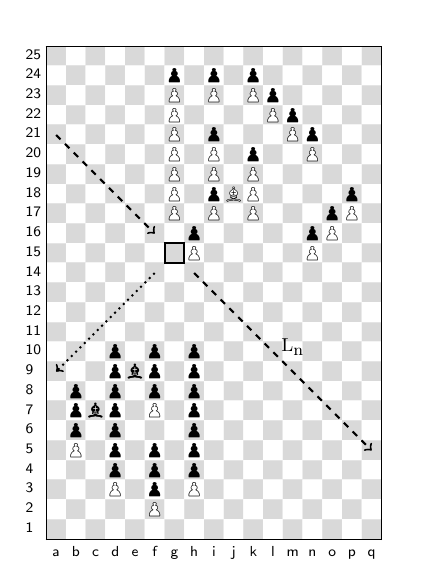}
        \caption{Pristine white node}
        \label{fig:White-Node-pristine}
    \end{subfigure}
    \begin{subfigure}{0.49\textwidth}
        \centering
        \includegraphics[width=1.13\linewidth]{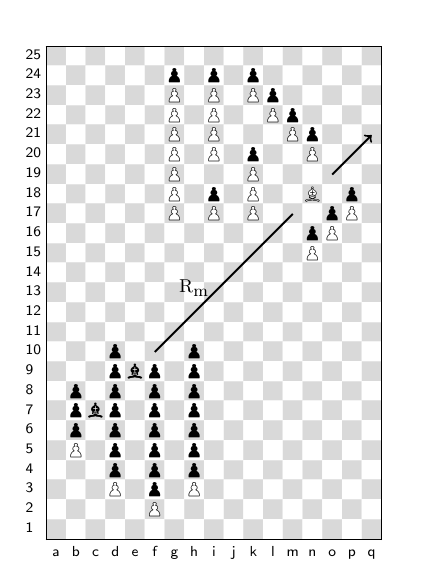}
        \caption{Used white node (White to move)}
        \label{fig:White-Node-used}
    \end{subfigure}
    \caption{White Node. The sequence {\tt 1...Bg15 2.\ Bxg15 hxg15 3.\ h16 g14 4.\ h17 g13 5.\ h18 g12 6.\ h19 g11 7.\ h20 g10 8.\ hxi21 g9 9.\ i22 g8 10.\ i21 gxf7 11.\ i20 f6 12.\ Bi19 f7 13.\ Bk21 f8 14.\ Bn18 f9} in (a) leads to the position in (b).}
    \label{fig:White-Node}
\end{figure}

\begin{figure}[!ht]
    \centering
    \includegraphics[width=\linewidth]{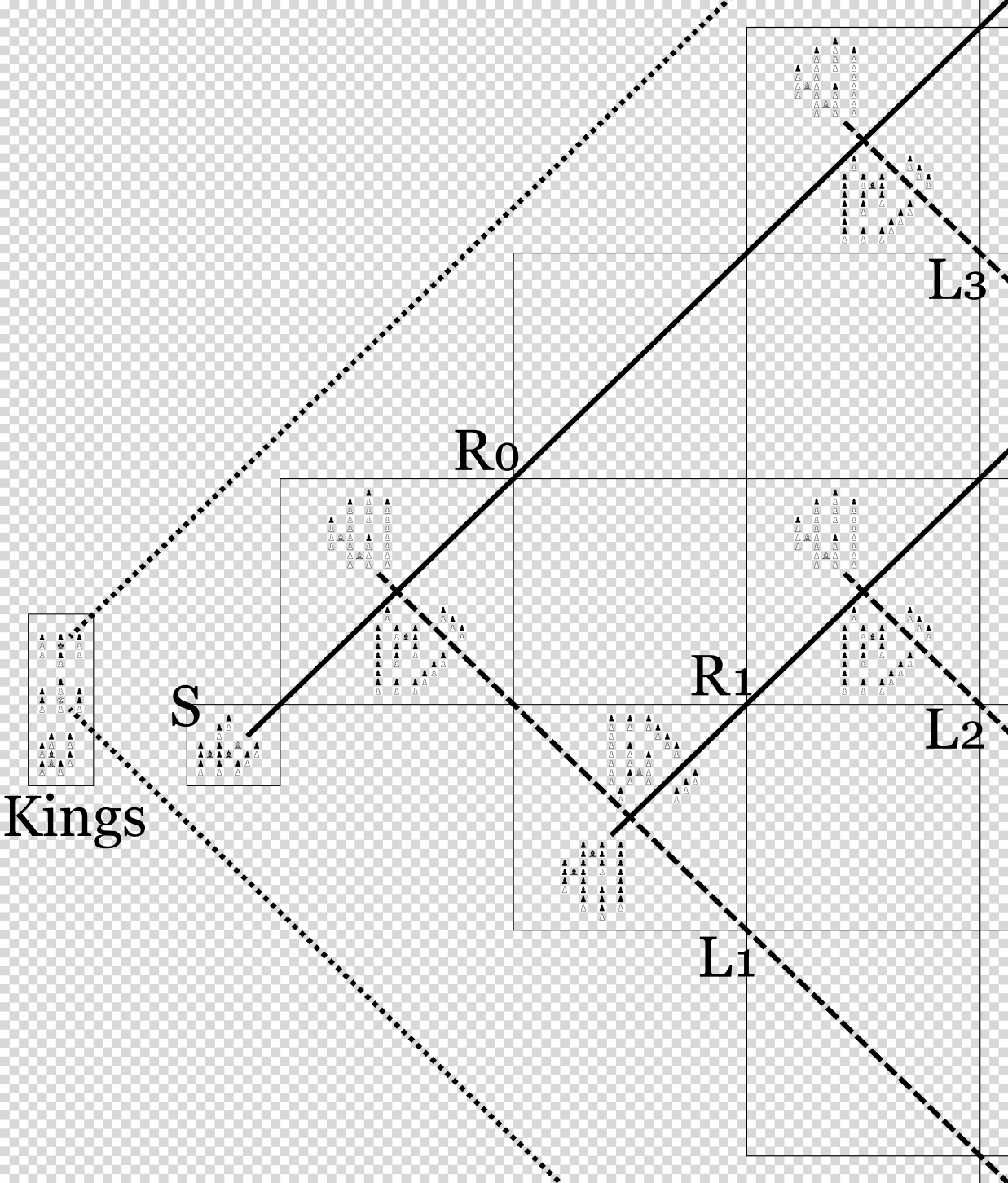}
    \caption{An open Gale-Stewart game with draws embedded onto the infinite chessboard. The spacing between adjacent parallel diagonals is $k=50$ here. The king chamber and root node were moved slightly to fit on the page, but this does not materially affect the analysis.}
    \label{fig:infinite-game-tree}
\end{figure}

All the guarded squares along the dashed diagonal are contained in so-called white nodes as shown in Figure~\ref{fig:White-Node-pristine}. Such a node is simply an upside-down and color-inverted version of a black node. Its function is the same but with the colors reversed. Black now threatens mate-in-one with his bishop on {\tt g15} by moving along the dotted diagonal to the bottom left and, thus, White is forced to capture it with his newly liberated bishop. This will once again kick off a series of forced moves, exactly mirroring the main line of the black node. The position after these moves is depicted in Figure~\ref{fig:White-Node-used}, with the white bishop standing on {\tt n18} and being threatened by the newly liberated black bishop on {\tt e9}. This forces the white bishop to move to the Northeast along the diagonal marked with the solid line and select a guarded square on it at some black node. Similarly to before, this freedom by White to choose a destination square with his bishop motivates the name `white node'.

These are all the necessary building blocks to embed open Gale-Stewart games with draws into the infinite chessboard as shown in Figure~\ref{fig:infinite-game-tree}. The diagonals going from the bottom left to the top right marked with a solid line are the ones where a white bishop flees from a black bishop onto a black node. Similarly, black choice diagonals are marked with dashed lines and go from the top left to the bottom right. The king chambers and root node are on the left of the position while the tree is on the right, with the nodes placed with $k=50$ as shown in Figure~\ref{fig:bishop-tree-structure} and the surrounding discussion.

We have thus shown how to construct a full bi-colored tree $\mathbb N^{<\omega}$ with alternately colored nodes on the chessboard. Any arbitrary desired game $G_{S_1,S_2}$ is obtained by only building a subset of the nodes of $\mathbb N^{<\omega}$ corresponding exactly to the reduced game tree $T$ of that game. A game in this position consists of White and Black alternately choosing a node to move to on their diagonals $R_n$ and $L_n$ respectively, while venturing deeper into the tree, until a player reaches a terminal node. A terminal node in this chess tree simply corresponds to the fleeing bishop having no safe square along his escape diagonal, resulting in that player no longer being able to protect his bishop from immediate capture while also preventing checkmate. Thus, a white terminal node hands Black the victory and vice versa. By construction, white terminal nodes in $T$ are losing nodes for player one in $G_{S_1,S_2}$, while black terminal nodes in $T$ are losing nodes for player two.

All other moves are of a local nature and completely forced. As already shown, any deviation from the main line in a node will result in the opponent winning in some small and constant number of moves. Thus, the chess position corresponds exactly to the open Gale-Stewart game with draws $G_{S_1,S_2}$ and the winner of the chess game corresponds to the winner of that Gale-Stewart game. Note that in the main line, both pristine and used nodes admit no local moves when not in use, so our earlier analysis of the deviations was indeed exhaustive.

This construction already takes the fifty-move rule into account since the main line involves very frequent pawn moves and piece captures. Thus, no further modifications are needed to accommodate the rule.

We are now ready to prove Theorem~\ref{thm:main}:

\begin{proof}[Proof of Theorem~\ref{thm:main}]
    First, we prove statement (i). Given the open Gale-Stewart game with draws $G_{S_1,S_2}$, let $T$ be its reduced game tree --- see Definition~\ref{def:reduced-game-tree}. Let $p$ be the infinite chess position corresponding to $T$ constructed by the procedure described in this section, namely by appropriately placing white and black nodes onto the intersections of diagonals if and only if they correspond to a node contained in $T$. Since $T$ is in $\Delta_1^{0,(S_1, S_2)}$ by Observation~\ref{obs:reduced-game-tree-prop}, the same holds for $p$ by the computability of the chess construction.

    Next, we prove statement (iii). Let $\sigma,\tau: \mathbb N^{<\omega} \to \mathbb N$ be strategies for $G_{S_1,S_2}$ and suppose player one plays according to $\sigma$ and player two according to $\tau$. Recall that by Observation~\ref{obs:reduced-game-tree-prop} and the fact the game is open, player one (resp.\ two) wins if and only if the resulting game reaches a player two (resp.\ one) losing node in finite time. Thus, the result of playing $\sigma$ against $\tau$ only depends on $\sigma|_T$ against $\tau|_T$ where the strategies are restricted to only those nodes contained in the reduced game tree.

    In the chess position $p$, we identify White as player one and Black as player two. The players are forced to climb through the reduced game tree with their bishops and as already shown in the chess analysis of this section, any deviation by either player from the main line in a node allows the other player to win in a uniformly bounded number of moves by following a computable strategy. We identify moves in $G_{S_1,S_2}$ restricted to sequences in $T$ with the choice of long-range bishop moves in the play stemming from $p$. The game is decided upon reaching a terminal node of $T$, keeping in mind that white terminal nodes in $p$ correspond to losing nodes for player one in $G_{S_1,S_2}$, while black terminal nodes in $p$ correspond to losing nodes for player two. Thus, White can force a win (resp.\ draw) in the chess game if and only if $\sigma|_T$ is a winning (resp.\ drawing) strategy, and vice versa, Black can force a win (resp.\ draw or win) in the chess game if and only if $\tau|_T$ is a winning (resp.\ drawing) strategy. In particular, White (resp.\ Black) can do so computably given $\sigma$ (resp.\ $\tau$) because of the computable nature of all required chess moves beyond the long-range bishop moves.
    
    Conversely, suppose a player has a winning (resp.\ drawing) strategy $\sigma$ in the chess position $p$. By the chess analysis of this section, this strategy prescribes that player's long-range bishop moves between nodes on the chess board when the opponent follows the main line. By the computability of the chess position in statement (i), this immediately yields a $\Delta_1^{0,\sigma}$ winning (resp.\ drawing) strategy $\hat\sigma$ for the corresponding player in the Gale-Stewart game restricted to its reduced game tree $T$. By the observation above that a strategy $\sigma$ is winning (resp.\ drawing) for some player if and only if that strategy restricted to $T$ is winning (resp.\ drawing), $\hat\sigma$ can be arbitrarily extended to a winning (resp.\ drawing) strategy for $G_{S_1,S_2}$.
    
    Finally, we prove statement (ii). We have already shown the correspondence of the forced outcome of the game for White (resp.\ Black) and player one (resp.\ two) in the proof of statement (iii). The game value of $G_{S_1,S_2}$ for some player (if it exists) cannot be larger than the game value of $p$ because while the theoretical outcomes are the same, a single turn in $G_{S_1,S_2}$ corresponds to multiple moves in the chess game: a long-range bishop move as well as some moves within a node.
\end{proof}

This theorem can be viewed as a strong universality statement about infinite chess, namely that it is at least as powerful as any other open game allowing draws. Notably, the explicitness of the method of the embedding immediately implies that any computable tree can be embedded as a computable chess position --- the embedding has exactly the same complexity as the reduced game tree $T$. Further, the complexities of the winning and drawing strategies of the chess game mirror exactly those of the strategies in the corresponding games being simulated. Thus, many results on the complexities of winning strategies in open games are applicable to infinite chess in their strongest forms.

\begin{figure}[!ht]
    \centering
    \includegraphics[scale=0.7]{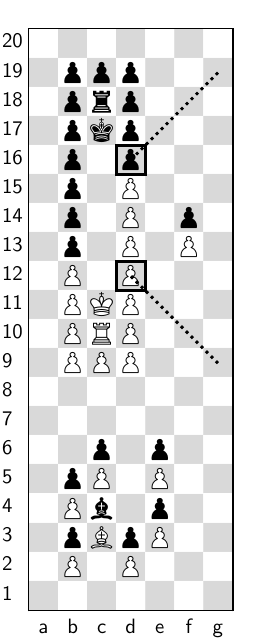}
    \caption{Selfmate King Chambers}
    \label{fig:King-Selfmate-Chambers}
\end{figure}

We now briefly discuss a modification of this construction for another style of chess puzzle. 

\begin{remark}The usual goal of a player in a chess puzzle is to checkmate the opposing king. However, a popular category of chess puzzles is \emph{selfmate} puzzles, where each player has the goal of causing the checkmate of their own king instead. Selfmate chess is also known as misère chess. Our construction can be modified to work in this setting. In particular, Theorem~\ref{thm:main} and its corollaries in Section~\ref{sec:intro} also hold for misère infinite chess.
\end{remark}
\begin{proof}
    Simply replace the king chambers of Figure~\ref{fig:King-Chambers} with those shown in Figure~\ref{fig:King-Selfmate-Chambers}. 
If White is successful in escaping the tree first via a dotted diagonal in a black node and moves his bishop onto the top right diagonal targeting the black {\tt d16} pawn in the king selfmate chambers, then Black cannot prevent selfmate via the forcing sacrifice {\tt Bxd16+ Kxd16\#} anymore. The same holds for the case where Black can escape via a dotted diagonal in a white node first and moves his bishop onto the bottom right diagonal targeting the {\tt d11} pawn. The blockading pawns on {\tt f13} and {\tt f14} guarantee that the vulnerable pawns cannot be guarded in time from any direction other than the indicated selfmate diagonals. Otherwise, the position plays out in exactly the same manner as in the construction for normal infinite chess, with both players racing each other down the tree with the goal of having their bishop escape or forcing their opponent to a terminal node.
\end{proof}

\section{Arbitrary Game Values Using Only Kings and Pawns}
\label{sec:pawn-tree}
The previous construction relied heavily on a single bishop's ability to select between infinitely many squares. One may then ask about the situation in versions of infinite chess where all pieces have bounded movement. Our main concern will be the following rule set: 
\begin{definition}
    \emph{Short-range infinite chess} is a version of infinite chess where all pieces are restricted to only be able to move a maximum distance of seven squares per turn along any direction. All other rules of infinite chess stay the same.
\end{definition}

This rule set was the one considered by Kőnig in his 1927 paper \cite{konig}, making it the first version of infinite chess in the literature that we are aware of.

Our main construction in this section will use only kings and pawns to realize all countable game values, via the embedding of a certain game on trees we call \emph{choosing-nodes-from-$Z$}.

\subsection{Zugzwang and Choosing-Nodes-From-Z}
\label{sec:zugzwang-and-choosing-nodes-game}

The phenomenon of transfinite game values does not arise when playing short-range infinite chess with only finitely many pieces, as the game tree is only finitely branching. However, the situation with infinitely many pieces is different. While no given piece has an infinite movement option, the players may still have an infinite choice, namely the choice of what piece to move. Thus there is the possibility of constructing positions with transfinite game values.

\begin{figure}[!ht]
    \centering
    \includegraphics[width=\linewidth]{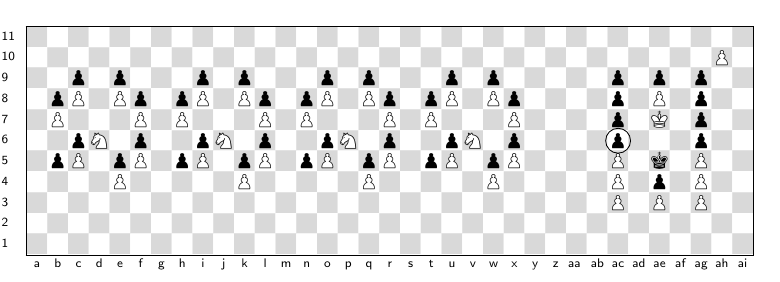}
    \caption{Short-range chess position with game value $\omega$ (Black to move). The pictured pattern continues infinitely to the left.}
    \label{fig:short-range-mate-in-omega}
\end{figure}

\begin{example} \label{ex:short_range_omega}
The position pictured in Figure~\ref{fig:short-range-mate-in-omega} consists only of pieces of finite range and has game value $\omega$ with Black to move.

\end{example}
\begin{proof}
 The kings of the two players are permanently deadlocked in their chamber on the right. White's only hope to checkmate Black is to reach the king chamber with one of his infinitely many knights, which are all currently imprisoned and cannot move more than twice unless Black allows it. However, Black is to move and is in zugzwang i.e., he would love to pass the turn and maintain the knight prisons, but the compulsion to move is his undoing. The only mobile black pieces in the position are the black pawns in the knight traps, each of which would free a white knight if moved. For example, if Black plays {\tt b4}, White can respond with {\tt Nc4} and the newly liberated knight will slowly make its way toward the circled pawn on square {\tt ac6}, checkmating the black king. Black cannot avoid freeing a knight in this manner, nor can he hope for stalemate, but he can at least freely choose which white knight to liberate and what its initial distance to the king chamber will be. Thus, the overall position has game value $\omega$.
\end{proof}

In Example~\ref{ex:short_range_omega}, the fact Black was in zugzwang to start played a pivotal role. In fact, one can prove zugzwang themes are necessary to any position with transfinite game value, by studying a further slight variation of the rules: 

\begin{definition}
    \emph{Short-range infinite chess with passing for Black} is a rule variant of short-range infinite chess, where Black may choose to pass his turn instead of moving a piece. Passing does not reset the fifty-move rule counter, and if Black's only legal move is to pass, then the position is still considered a draw by stalemate. All other rules are as before.
\end{definition}

Our interest in this variant is due to the following theorem:

\begin{theorem}
\label{thm:short-range-zugzwang}
    Only finite game values arise for White in short-range infinite chess with passing for Black.
\end{theorem}
To prove this, we need the following lemma:

    \begin{lemma} \label{lem:srz-finiteness-lemma}
        If a position $q$ in short-range infinite chess with passing for Black has game value $k < \omega$ for White, then there exists a finite subset $S(q)$ of the board $\mathbb Z^2$ such that any position having the same arrangement of pieces on $S(q)$ as $q$ and the same or smaller fifty-move rule counter as $q$, has game value at most $k$ for White.
    \end{lemma}
    \begin{proof}
        
     We shall prove this statement by induction on $k$ and the player to move, using the fact that the maximum move distance of all pieces on the board is bounded by $7$.
    
    When $k=0$, by definition Black is to move and is checkmated, which only depends on the pieces in a finite area around Black's king. 
    
    If White is to move in position $q$ of game value $k$, then by definition White may move to some position $q'$ of game value $k-1$ (Note that this move might need to be very far away from Black's king due to a need to reset the fifty-move rule counter). We then may take $S(q)$ to be the union of the squares in $S(q')$ together with those finitely many squares involved in the legality of White's move, noting that these squares must include a finite area around White's king to ensure it is not in check after the move. Then from any position matching $q$ on $S(q)$, White may play the same move as in $q$ to reach a position matching $q'$ on $S(q')$, which is then mate in $k-1$ by the definition of $S(q')$.
    
    If Black is to move in a position $q$ of game value $k$, let $q'$ be the position reached by a Black pass and $\{q_i\}$ the finite collection of positions Black may move to by changing the arrangement of the pieces on $S(q')$. Then we may take $S(q)$ to be all squares within a single move's distance of $S(q')$ or one of the $S(q_i)$, as well as all squares involved in the legality of some arbitrarily chosen non-pass move for Black. $S(q)$ is then certainly finite. If a position agrees with $q$ on $S(q)$, then Black is not in stalemate and so must move or pass. Black moves not involving the squares in $S(q')$ will result in a position matching $q'$ on $S(q')$, while Black moves involving it necessarily lead to a position matching one of the $q_i$ on $S(q_i)$. Thus the game value of this position is also at most $k$. This concludes the induction.
\end{proof}
    We now prove Theorem~\ref{thm:short-range-zugzwang}.
    \begin{proof}[Proof of Theorem~\ref{thm:short-range-zugzwang}]
    We may assume White restricts play to only moves which reduce the game value. In this way we need not worry about the threefold repetition rule and so may ignore how we reached a given position without worry. Suppose for the sake of contradiction that there existed a position $p$ with game value $\alpha \ge \omega$. Then, by an easy induction, in play stemming from $p$ Black may force a position $p'$ of game value $\omega$. By the definition of game values, Black is necessarily to move in $p'$ as $\omega$ is not a successor ordinal. We are thus reduced to showing no position $p'$ with game value $\omega$ and Black to move can exist.
    
    As a Black pass in $p'$ results in a position $p''$ of game value $k < \omega$ by definition, Lemma~\ref{lem:srz-finiteness-lemma} grants us a finite subset $S(p'')$ of the board such that any position agreeing with $p''$ on $S(p'')$ has game value bounded by $k$. However, as $p'$ agrees with $p''$ on $S(p'')$ and Black has at most finitely many moves which modify the pieces on $S(p'')$, Black can move to at most finitely many positions with game value greater than $k$. As these positions will still have finite game value, we have a contradiction, as $\omega$ is not the supremum of finitely many natural numbers. Thus, no such $p'$ and hence no such $p$ may exist.
\end{proof}
The reason for considering the game with passes allowed for Black in this particular manner is that we have only improved Black's prospects, so in particular we may deduce facts about game values in the game without passes.
\begin{corollary}
If a position $p$ in short-range infinite chess has game value $\alpha \ge \omega$ for White, then $p$ is not winning for White in short-range infinite chess with passing for Black.
\end{corollary}
\begin{proof}
    As $\alpha \ge \omega$, Black can survive for any finite number of turns without passing. Therefore by Theorem~\ref{thm:short-range-zugzwang}, Black has some way to survive forever when allowed to pass.
\end{proof}
In other words, zugzwang is the only possible mechanic we can use for constructing transfinite game values in short-range infinite chess. Morally, the compulsion to move is the only mechanic of short-range chess that can force the defending side to make a truly infinite choice --- all other rules are either local in nature or are draw rules which mostly hinder the attacking side. This observation is not unique to chess, and the broad strategy of modifying the rules to remove non-local behavior and then trying to prove some sort of compactness theorem can be applied to a number of infinite games. For example, for many desirable constructions in infinite go, a version of this proof reveals that non-local interactions can only be mediated via infinitely large groups or the ko rule. A similar argument appears in \cite{hamkins2023hex}, where it was shown that infinite hex and other simple stone-placing games are of an intrinsically local nature in the sense that winning positions are entirely decided by play in a finite area of the board, and thus do not admit transfinite game values.

With the restrictions stemming from Theorem~\ref{thm:short-range-zugzwang} in mind, we will nonetheless prove that short-range infinite chess is very powerful and, in particular, all countable ordinals arise as game values, even when restricting our chess arsenal to only pawns and kings:

\secondthm*

We will prove Theorem~\ref{thm:game-values-pawns} by embedding a collection of games into the infinite chessboard, which are simple enough to construct solely out of pawns and kings using the zugzwang idea, while being sufficient to exhaust all countably infinite game values:

\begin{definition}[\emph{Choosing-nodes-from-$Z$ game}] \label{def:choosing_from_z_game}
    Let $T$ be the full binary tree i.e., $T = \{0,1\}^{<\omega}$. Let $Z$ be an infinite subset of $T$. To $Z$ we associate the \emph{choosing-nodes-from-$Z$ game} as follows: On his turn, Black has to choose a node from $Z$ that he has not chosen before. White does not have any meaningful moves and is only watching Black make his moves. The game continues until Black has chosen two nodes over the course of the game such that neither node is an initial segment of the other. In that case, the game ends immediately and White wins. Otherwise, Black draws by infinite play i.e., by successfully naming infinitely many nodes from $Z$ that all lie on the same branch. 
\end{definition}

The requirement that $Z$ be infinite is not especially important, but will be needed to avoid stalemate concerns for Black once we begin embedding these games into chess. We call a node of the binary tree a $\emph{Z-node}$ if it is contained in the set $Z$. Nodes not in $Z$ are called \emph{normal nodes}. Any normal node which does not have any Z-nodes as descendants is irrelevant to the game, so we omit such nodes in all figures.

Since a win for White in this game only occurs at a finite stage, if it occurs, this \emph{choosing-nodes-from-$Z$} game is necessarily open, and so the theory of ordinal game values applies. We now show that arbitrarily large countably infinite ordinals arise as game values of the \emph{choosing-nodes-from-$Z$} game.

\begin{figure}[!ht]
    \centering
    \includegraphics[width=0.5\linewidth]{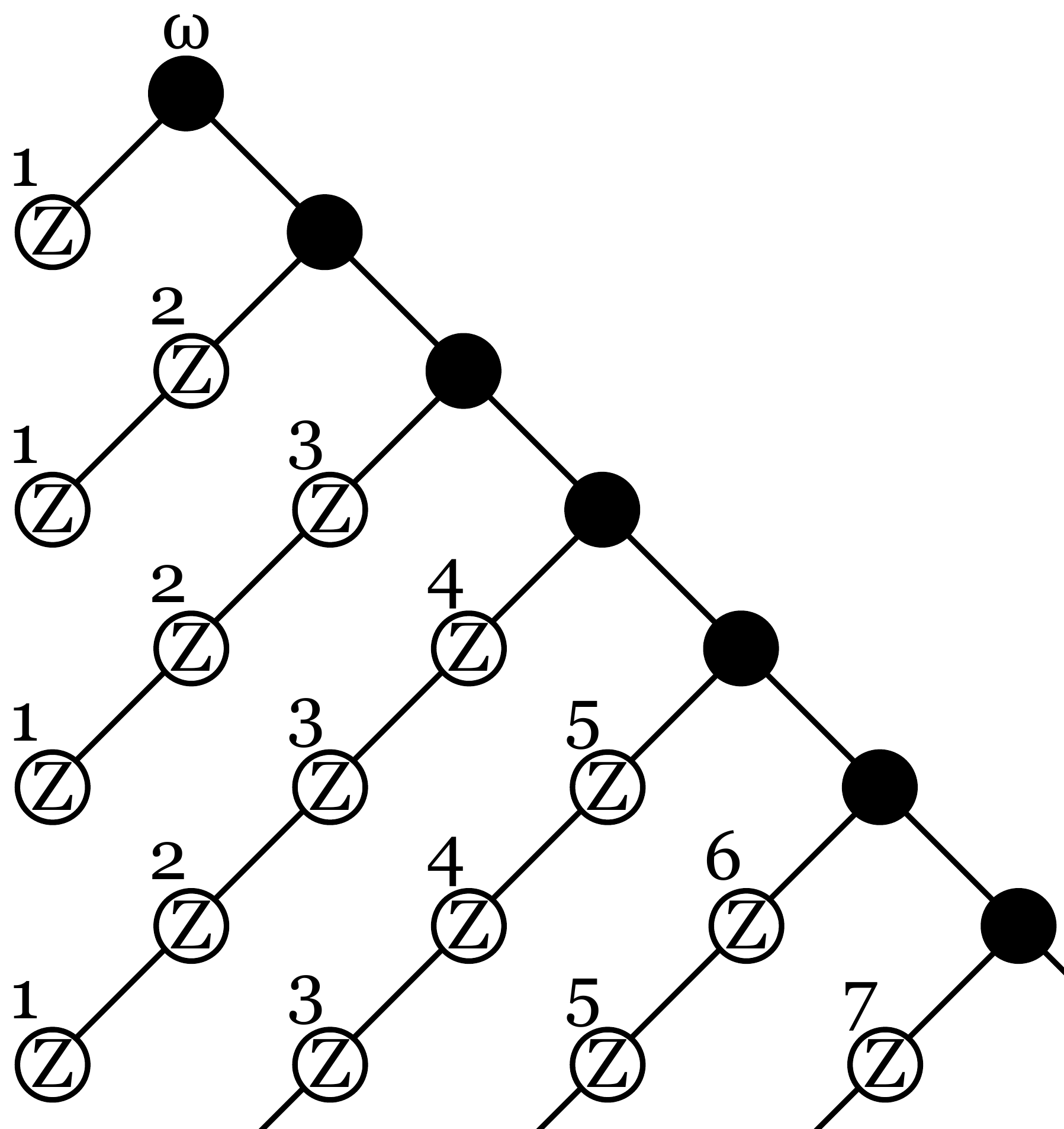}
    \caption{\emph{Choosing-nodes-from-$Z$} game with game value $\omega$}
    \label{fig:binary-tree-game-value-omega}
\end{figure}

\begin{lemma}
\label{lem:choosing-nodes-game}
    For any countable and infinite ordinal $\alpha$, there exists an infinite set $Z \subseteq \{0,1\}^{<\omega}$ such that the corresponding \emph{choosing-nodes-from-$Z$} game has game value for White at least $\alpha$.
\end{lemma}

\begin{proof}
For the base case of our transfinite induction, a \emph{choosing-nodes-from-$Z$} game with game value $\omega$ is shown in Figure~\ref{fig:binary-tree-game-value-omega}. The relevant parts of the binary tree are pictured and Z-nodes are denoted by the letter $Z$, while normal nodes are denoted by black circles. The root node is labeled with $\omega$ to signify the game value of the tree below it. The full tree contains a branch with one Z-node, a branch with two Z-nodes, a branch with three Z-nodes etc, as pictured. Crucially, these important branches are arranged in such a way that their pairwise intersections only contain normal nodes. On his first move, Black is forced to name a Z-node lying on one of these branches, thus committing to it in the sense that he will lose the game if he ever names a node lying on a different branch. Since Black was able to commit to a branch containing an arbitrarily large finite number of Z-nodes on his first move, the total game value of the position is $\omega$.

\begin{figure}[!ht]
    \centering
    \begin{subfigure}{0.2446\textwidth}
        \centering
        \includegraphics[width=\linewidth]{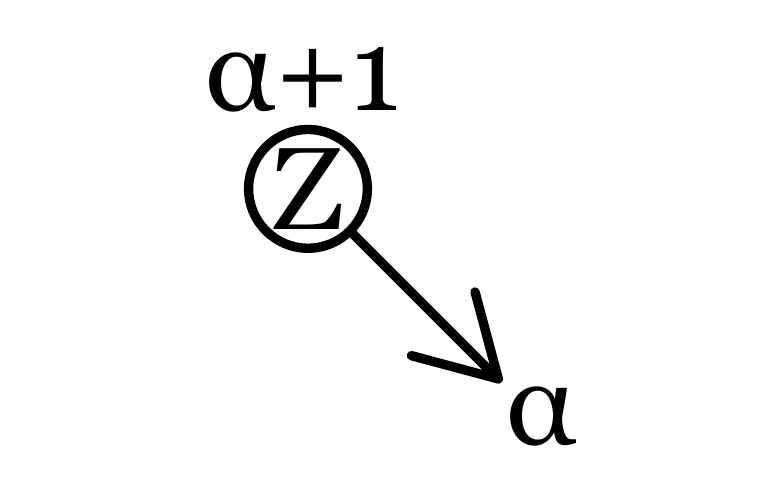}
        \caption{Successor ordinal}
        \label{fig:binary-tree-successor-ordinal}
    \end{subfigure}
    \;\;\;\;\;\;
    \begin{subfigure}{0.45\textwidth}
        \centering
        \includegraphics[width=\linewidth]{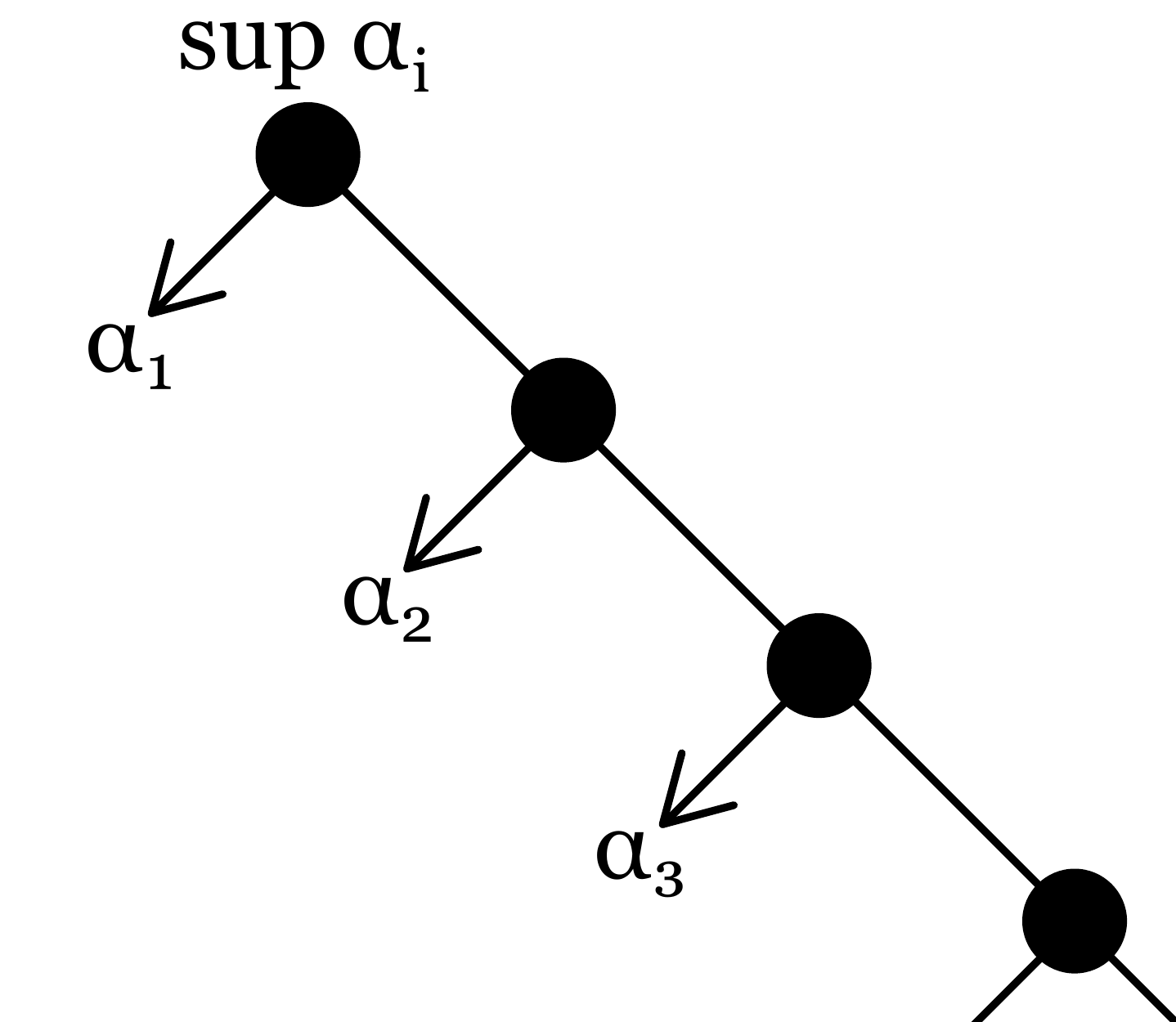}
        \caption{Limit ordinal}
        \label{fig:binary-tree-limit-ordinal}
    \end{subfigure}
    \caption{Construction of ordinal game values in the \emph{choosing-nodes-from-$Z$} game. The root node is labeled with a lower bound for the total game value of the tree, while subtrees are labeled by their respective game values.}
    \label{fig:binary-tree-ordinals}
\end{figure}

For the successor ordinal case, given some tree with corresponding game value $\alpha$, we may construct a tree with game value at least $\alpha+1$ as in Figure~\ref{fig:binary-tree-successor-ordinal}. A Z-node has been placed at the root of the tree, below which the game value $\alpha$ tree is embedded as a subtree. Black will clearly still lose this game, and the game value is at least $\alpha+1$ since Black can play at the root for his first move, resulting in a position of game value $\alpha$.

The limit ordinal case works in the same way as the game value $\omega$ construction. Given a countable collection of these games with the $i$th game having game value $\alpha_i$, we consider the tree depicted in Figure~\ref{fig:binary-tree-limit-ordinal}. As all of Black's moves except the losing one must take place in one of the subtrees corresponding to some $\alpha_i$, Black will still lose this game and so the position still has a game value. As Black is free to choose whichever $\alpha_i$ he pleases, the total game value is at least $\sup \alpha_i$. Therefore, arbitrarily large countably infinite ordinals are indeed realized as game values of the \emph{choosing-nodes-from-$Z$} game.
\end{proof}

\subsection{Infinite Chess Construction}
\label{sec:pawn-tree-chess}

We will now prove Theorem~\ref{thm:game-values-pawns} by turning our attention to the realization of the \emph{choosing-nodes-from-$Z$} games in infinite chess. As before, we will achieve this through the construction of various small gadgets which can be assembled together to form any instance of the \emph{choosing-nodes-from-$Z$} game. When doing the analysis of these gadgets we must consider cases where one player passes locally by playing some move somewhere outside the figure. We tacitly assume that such moves do not end the game prematurely, lead to infinitely long forced sequences, or cause a black piece to enter the relevant figure. These assumptions will be justified upon assembly of the final position. 

The branches of the tree we embed will be made out of pawn chains. \begin{definition}
    A \emph{pawn chain} is a configuration of white and black pawns on the chessboard such that \begin{enumerate}[(i).]
        \item In front of each pawn is a pawn of the opposite color blockading it.
        \item Every black pawn save the northernmost one is defended by a unique black pawn in the chain.
        \item No pawn of any color in the chain is defended or attacked by any black pawn outside the chain.
    \end{enumerate}
    We call the northernmost black pawn the \emph{base pawn} of the chain. If we replace one of the black pawns in a pawn chain by a blank square or a white pawn, we say the chain is \emph{activated}.
\end{definition}
An example of a pawn chain with base pawn {\tt b16} is shown in Figure~\ref{fig:Pawn-chain}. Pawn chains are useful due to the following properties:

\begin{lemma} \label{lem:pawn_chains}
    White can force a white pawn onto the square of the base pawn of any activated chain. Furthermore, White can also play moves elsewhere on the board between each move he plays in the chain, without damaging his ability to eventually force a pawn onto the base square.
    If Black is to move then Black can ensure that the only file White can create a free pawn on is the file of the base pawn.
\end{lemma}

\begin{proof}
    As the pawn chain is activated, White has a pawn in the chain which he may push, threatening one of Black's pawns, or he is already threatening one of Black's pawns. Black may choose to ignore the threatening pawn or take it. In the first case White may simply capture the threatened pawn, thus moving up the chain, while in the second case Black has freed a new white pawn, which now may move forward. In both cases, White can also choose to play moves elsewhere before returning to this pawn chain. Continuing up the pawn chain in this manner, White eventually lands a pawn on the square of the base pawn. Meanwhile, Black can ensure that this is the only way White can free a pawn by simply always choosing to capture when able.

\end{proof}

\begin{figure}[!ht]
    \centering
    \begin{subfigure}{0.40\textwidth}
        \centering
        \includegraphics[scale=0.8]{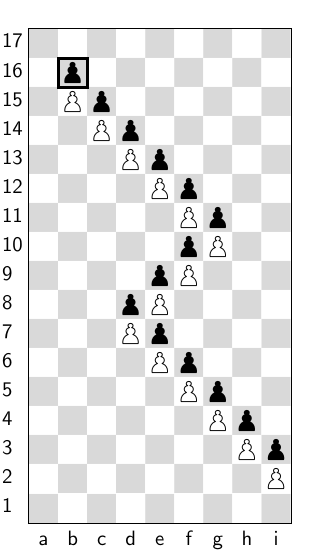}
        \caption{Pawn chain segment with base pawn b16}
        \label{fig:Pawn-chain}
    \end{subfigure}
    \;\;\;\;\;
    \begin{subfigure}{0.40\textwidth}
        \centering
        \includegraphics[scale=0.8]{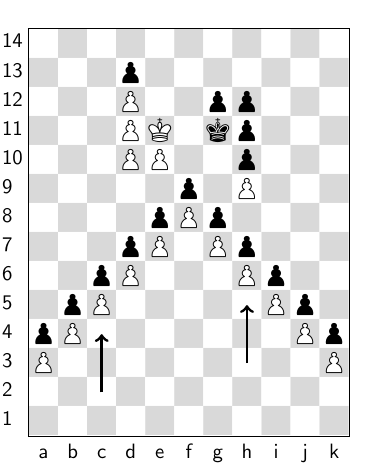}
        \caption{King chamber. The diagonal pawn chains continue forever.}
        \label{fig:King-chamber-pawns}
    \end{subfigure}
    \caption{Basic building blocks of the pawn tree}
    \label{fig:pawn-building-blocks}
\end{figure}

As an example, if the black pawn on {\tt i3} in Figure~\ref{fig:Pawn-chain} disappeared, Black would still be deadlocked, while White would have the option of advancing his pawn to {\tt i3} at any time. Now, no matter how Black reacts --- by capturing via {\tt hxi3} or by playing somewhere else on the board --- White will be able to move a pawn to {\tt h4} whenever he so chooses. In the following turns, White will be able to keep sacrificing his pawns in this manner to advance upward to {\tt g5, f6, e7, d8, e9,} and so on. Assuming that White does not care about the black pawns getting released downward, he can keep chewing his way upward through the pawn chain regardless of Black's play, eventually landing a pawn on {\tt b16}, the site of the base pawn.

The king chamber of the position is pictured in Figure~\ref{fig:King-chamber-pawns}. It is situated at the very top of the entire construction and features two deadlocked kings, neither of whom can make a move. Below the king chamber are two pawn chains with the same base pawn {\tt f9}, which we understand as continuing diagonally downward forever. We also require that the diagonal directly below each chain be empty. We will build the entire pawn tree below these main chains, and White's win condition will be to free a pawn from the tree in order to activate one of the king chamber's pawn chains.

\begin{lemma} \label{lem:bounded_king_chamber}
    Suppose neither player has any pawns above the pawn chains in the king chamber, besides those depicted in Figure~\ref{fig:King-chamber-pawns}. Then White can force a win if and only if White can force a white pawn onto one of the diagonals directly below these pawn chains. 
\end{lemma}

\begin{proof}
    If White ever has a pawn reach the necessary diagonal, then on his next turn he can force the situation of Lemma~\ref{lem:pawn_chains}, and so is guaranteed to eventually land a white pawn on {\tt f9}, with the unstoppable {\tt f10\#} next move. Since the rest of the position is below the game-deciding pawn chain and only consists of pawns, Black is unable to stop his defeat once a white pawn reaches the necessary diagonal. After all, he can only ever move pawns downward and thus he will never be able to influence the pawn chain. White also cannot hope to win if White cannot get a pawn onto one of these diagonals, since otherwise there is no way to free any pawn capable of checking the Black king. 
\end{proof}

The complete binary tree from \emph{choosing-nodes-from-$Z$} will be represented on the chessboard as a tree of pawns built below the game-deciding pawn chain of the king chamber. We have already seen the pawn chains which will form the edges of the tree, so we now construct components representing the nodes of this tree.
A normal binary node is pictured in Figure~\ref{fig:binary-node}, and can be understood as two pawn chains which overlap. Lemma~\ref{lem:pawn_chains} implies that if either of these pawn chains is active, then White can force a pawn to the base of the pawn chain that exits the diagram at {\tt a13}. Key to our construction will be the situation in which both of these pawn chains are activated below the node. 

\begin{figure}[!ht]
    \centering
    \includegraphics[scale=0.85]{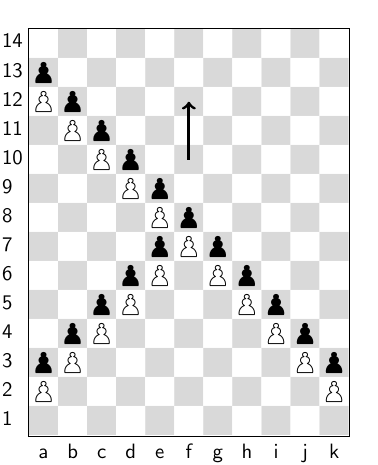}
    \caption{Binary node}
    \label{fig:binary-node}
\end{figure}

\begin{lemma} \label{lem:normal_node}
    Suppose the pawn chains of Figure~\ref{fig:binary-node} through {\tt k3} and {\tt a3} are both activated below {\tt f8}. Then White can force a pawn out of the top of the diagram along the f-file.
\end{lemma}

\begin{proof}
    White first marches up the two chains as in Lemma~\ref{lem:pawn_chains}, giving priority to the chain activated in the lower rank.  The purpose of giving priority to the lower chain is so that the black pawns released in the process are below the other chain and cannot influence it. White continues this procedure until either there are white pawns on both {\tt e7} and {\tt g7}, or there is a white pawn on one of {\tt e7} and {\tt g7} and Black has played one of {\tt fxe7} or {\tt fxg7}, leaving {\tt f8} vacant. White then may play {\tt exf8} in the first case and {\tt f8} in the second, in both cases threatening to push {\tt f9} and subsequently escape out of the top of the diagram. If Black delays matters with {\tt exf8}, then in both cases White has a move to recapture the now unguarded pawn on {\tt f8}, after which escape along the f-file is unstoppable.
\end{proof}

Our goal is now to arrange the position so that White liberating any pawn from a normal node along its f-file guarantees White's victory. By Lemma~\ref{lem:bounded_king_chamber}, White will win if he can march a pawn to the game deciding pawn chain at the top of the position. As there may be other pawn chains above the node where the pawn is released, to allow such a march we need a mechanism that lets a white pawn below a pawn chain release another white pawn above the chain. This is achieved by the position of Figure~\ref{fig:pristine-white-pawn-break}, as we now argue. The idea is to have the first pawn free a large enough horde of pawns under the chain, which can then force their way through.

\begin{figure}[!ht]
    \centering
    \begin{subfigure}{0.49\textwidth}
        \centering
        \includegraphics[scale=0.85]{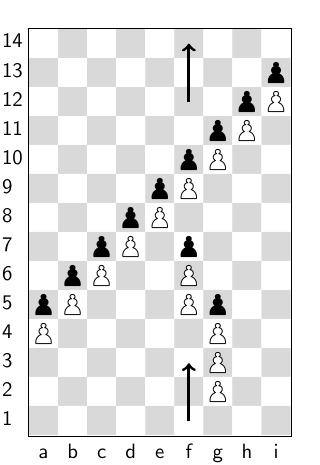}
        \caption{Pristine}
        \label{fig:pristine-white-pawn-break}
    \end{subfigure}
    \begin{subfigure}{0.49\textwidth}
        \centering
        \includegraphics[scale=0.85]{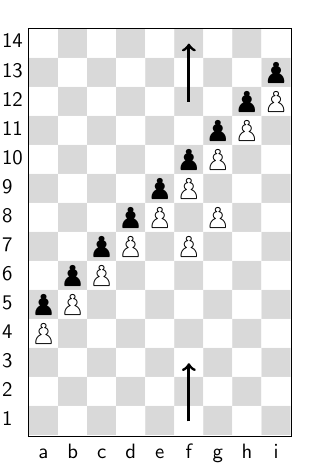}
        \caption{In use}
        \label{fig:in-use-white-pawn-break}
    \end{subfigure}
    \caption{White pawn crossing mechanism}
    \label{fig:pawn-breaks}
\end{figure}

\begin{lemma} \label{lem:crossing}
    If a white pawn enters the position in Figure~\ref{fig:pristine-white-pawn-break} from below along the $f$-file, then white can force a white pawn out of the top of the figure via the $f$-file.
\end{lemma}

\begin{proof}

Having entered via {\tt f1}, White's new pawn will march unopposed to {\tt f4}, after which White can guarantee the capture of Black's {\tt g5} pawn on one of {\tt g5} or {\tt f4}. This leaves White with at least two free pawns on the g-file. By moving up the pawn column on the g-file, White can similarly force the capture of Black's {\tt f7} pawn on {\tt f7} or {\tt g6}. White has now removed all Black pawns below the pawn chain from the board, and is left with at least one pawn on the f-file and one pawn on the g-file (indeed the White pawns initially on {\tt f6} and {\tt g2} remain on their respective files in all cases). Thus, White can force the position of Figure~\ref{fig:in-use-white-pawn-break}, with perhaps some additional white pawns below the important pawns on {\tt f7} and {\tt g8}.

From Figure~\ref{fig:in-use-white-pawn-break}, the argument proceeds as in Lemma~\ref{lem:normal_node}. Indeed, the position can be understood as a normal node with very small pawn chains, one through {\tt f8}-{\tt e9}-{\tt f10} that is currently activated at f8, and another through {\tt g9} and {\tt f10} that is currently activated at {\tt g9}. Exactly as in Lemma~\ref{lem:normal_node} White wins by pushing {\tt f8}, moving a pawn to {\tt e9}, then moving a pawn to {\tt g9}. White now breaks through by moving some pawn to {\tt f10}, threatening {\tt f11}. Should black play {\tt gxf10}, White simply captures back on {\tt f10}, with {\tt f11} next move being unstoppable. The {\tt f11} pawn now may escape out of the top of the figure along the f-file.
\end{proof}

This concludes the introduction of most necessary components for the embedding of a \emph{choosing-nodes-from-$Z$} game $G$ into the infinite chessboard. We now give the recipe for constructing a chess position $p(G)$ embedding $G$, in the process introducing the mechanism by which we mark nodes as Z-nodes.

\begin{figure}[!ht]
    \centering
    \includegraphics[width=\linewidth]{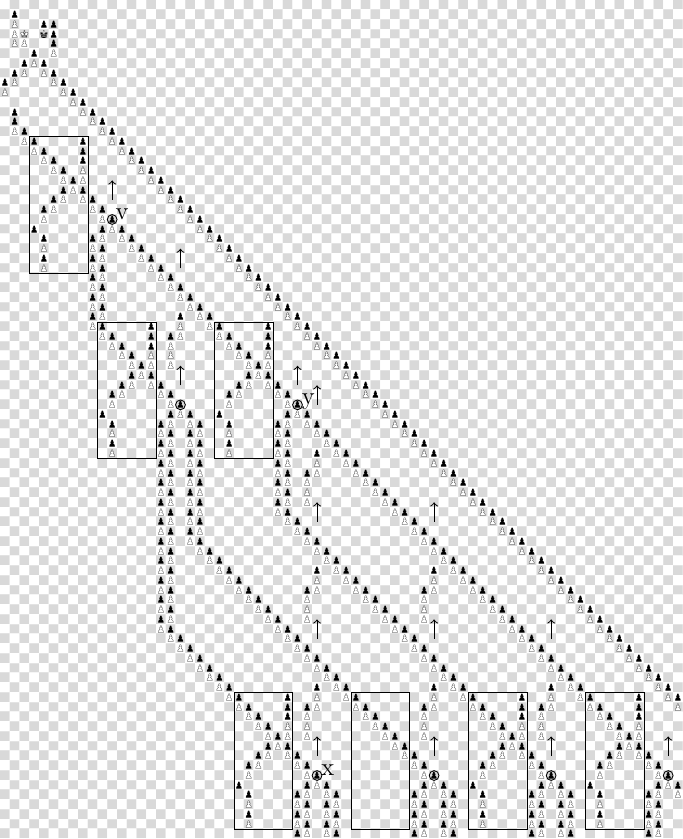}
    \caption{Pawn tree from the root node until nodes of depth $2$. Nodes are marked by circled blacked pawns. The second node of depth 2 from the left is a normal node and so has a normal pawn chain leaving it, while all others are Z-nodes and hence are equipped with a copy of figure~\ref{fig:Z-component-unused}, the Z-component.}
    \label{fig:pawn-tree}
\end{figure}

The position $p(G)$ will begin with Black to move. The global arrangement of the components on the chessboard is demonstrated in Figure~\ref{fig:pawn-tree}, which illustrates the placement of all nodes of depth 2 or less. The top of the position is given by the king chamber and game-deciding pawn chains. Below that, the full binary tree is built out of pawn chains. Each circled square belongs to a binary node as in Figure~\ref{fig:binary-node}, and the pawn chains from every node go all the way up to the root of the tree, which we have made a dead end by including an extra black pawn behind the base pawn to catch any freed white pawns. For every Z-node, the pawn chain leaving it towards the north is equipped with an extra component indicated by the regions in rectangles in the figure, which will be introduced below.

The tree is arranged so that every binary node is on a distinct file, there are only diagonal pawn chains above every node, and the pawn chains above each node are sufficiently spaced out. By doing this, all pawn chains above a node can be equipped with their own copy of the white pawn crossing mechanism introduced in Figure~\ref{fig:pristine-white-pawn-break}. In Figure~\ref{fig:pawn-tree} we illustrate one way of achieving this, in which nodes of higher depth are at least two files to the right of nodes of lower depth, and nodes of the same depth are placed at the same height, spaced out every 12th file. The flexibility of the pawn chains means that many such constructions are possible.

\begin{figure}[!ht]
    \centering
    \begin{subfigure}{0.45\textwidth}
        \centering
        \includegraphics[scale=0.8]{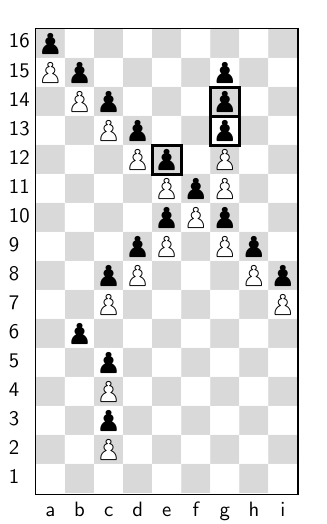}
        \caption{Pristine}
        \label{fig:Z-component-unused}
    \end{subfigure}
    \begin{subfigure}{0.45\textwidth}
        \centering
        \includegraphics[scale=0.8]{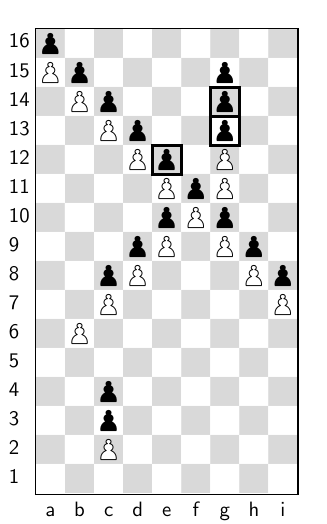}
        \caption[]{After {\tt 1...b5 2.\ cxb5 c4 3.\ b6}}
        \label{fig:Z-component-activated}
    \end{subfigure}
    \caption{The Z-component. A Black move here grants White an activated pawn chain and leaves Black with no local moves.}
    \label{fig:Z-component}
\end{figure}

As mentioned, we mark nodes as Z-nodes by including a copy of the mechanism of Figure~\ref{fig:Z-component-unused} on the edge immediately above the node we wish to mark (the boxed regions to the left of the nodes in Figure~\ref{fig:pawn-tree}). We call this mechanism a Z-component. Such a component essentially consists of two pawn chains that meet up on square {\tt f11} of Figure~\ref{fig:Z-component-unused} - a long pawn chain that is part of the tree and a very short local pawn chain coming from {\tt c8}, as depicted. However, in contrast to the binary nodes, this merger of two pawn chains is not a source of danger for Black, as there are extra black pawns on {\tt g13}, {\tt g14}, and {\tt g15}, which prevent any white pawn from escaping along the f-file. Note that Black has a move immediately available in every Z-component with his respective {\tt b6}-pawn but nowhere else, while White has no moves available anywhere in the position. We will see in the proofs of Lemmas \ref{lem:pawn-tree-sim1} and \ref{lem:pawn-tree-sim2} that Black will be in almost constant zugzwang throughout the chess game and that a Black move in a previously unused Z-component will act like a Black move in $G$. If Black ever loses the corresponding game $G$, White will be able to force checkmate in finitely many moves. We will also see that White cannot do better than this.

We shall now analyze the play in such a position. Due to the requirement that the set $Z$ be infinite, there are infinitely many Z-components and the game can never end with Black being stalemated. The fifty-move rule and threefold repetition also are irrelevant as both players make only pawn moves. As a result, the game can only be ended by checkmate or White being stalemated.

We first argue that the game value of $p(G)$ is not less than the game value of $G$.
\begin{lemma}
\label{lem:pawn-tree-sim1}
Suppose $G$ is a game of choosing-nodes-from-Z of game value $\alpha$.
If White is winning in $p(G)$, then the game value of $p(G)$ is at least $\alpha$.
\end{lemma} 
\begin{proof}
We suppose that Black begins the game by playing $G$ and $p(G)$ simultaneously, trying to play the sequence {\tt 1...b5 2.\ cxb5 c4 3.\ b6} in the copy of Figure~\ref{fig:Z-component-unused} corresponding to his chosen move in $G$ (So every Black move in $G$ corresponds to two on the chessboard). To win the chess game, White eventually must deviate from this sequence in some Z-component, for otherwise play would continue indefinitely. As White has no other moves in the current Z-component and no moves outside of previously used Z-components, White's only option for such a deviation is to play {\tt b7} in a Z-component where the exchange {\tt 1...b5 2.\ cxb5 c4 3.\ b6} has already occurred, shown in Figure~\ref{fig:Z-component-activated}. Suppose White does this before Black has lost $G$. We argue that Black now has a drawing strategy, and thus the game value of $p(G)$ (if it exists) is at least $\alpha$. 

Black's defense will be given by a simple algorithm. Black will immediately respond to White's {\tt b7} with {\tt cxb7}, which importantly makes the waiting move of {\tt b6} immediately available to him. Black now announces that he will capture a White pawn every turn if possible, unless that capture would be with a pawn on one of the highlighted squares {\tt e12}, {\tt g13}, and {\tt g14} in a Z-component. If White makes a move Black cannot or will not capture, Black will instead make a waiting move by pushing some pawn that White cannot immediately capture. We argue that Black is always able to follow this strategy, and that by doing so Black will eventually stalemate White.

Since Black has not yet lost the associated \emph{choosing-nodes-from-$Z$} game, any two activated pawn chains of the tree meet in the Z-component where the shorter pawn chain originates. In this component, Black's strategy ensures that he will eventually get to play {\tt fxe10} or {\tt fxg10}, depending on which direction White approaches from. In the former case, the lower right pawn chain (if activated) now can technically free a pawn along the g-file, but it immediately gets stuck at {\tt g10}. In the latter case, the {\tt e11} pawn will never move under Black's strategy, so while the lower left pawn chain (if activated) can free a pawn on the e-file it will similarly get stuck at {\tt e10}. In both cases White can choose to play {\tt f11} followed eventually by the exchange {\tt fxe12 dxe12}, activating the upper left pawn chain and continuing up the tree, but this is a pawn chain which was already activated before the start of this sequence anyways. If White plays {\tt f12} instead of {\tt fxe12}, the newly freed pawn will eventually get stuck at {\tt g13} or {\tt g14}, or be captured at {\tt f14}. All in all, at each Z-component White must choose a pawn chain to make a dead end, cannot permanently free any pawns, cannot activate any new pawn chains, and can force Black to play at most 5 waiting moves (to respond to {\tt e10} or {\tt g10} and up to 4 moves with the f pawn).

Note that the black waiting move {\tt b6} White made available with the initial deviation is enough to pass the turn back to White if he were to continue the sequence in the Z-component where he initially ignored Black. The only other times Black needs a waiting move available are when White plays a move that cannot be answered with a capture in a Z-component or at the root of the entire tree, but these come only after White has released a horde of Black pawns by climbing up a pawn chain to reach the relevant node.

White's prospects are therefore grim. Any pawn chain he plays in will slowly climb up the tree until running into the dead-end root node or into another pawn chain at a Z-component. As Black is able to wait out any White moves that are not in pawn chains, White is forced to continue playing in pawn chains until all of the finitely many pawn chains of finite length Black had activated are exhausted. At this point, White is left with no moves and is stalemated. 

White therefore cannot win before Black loses $G$, meaning the game value of this position must be at least $\alpha$ if it exists.

\end{proof}

It remains to show that $p(G)$ is winning for White if $G$ is winning for White, so that the game value of $p(G)$ exists.

\begin{lemma}
\label{lem:pawn-tree-sim2}
    Suppose $G$ is a game of choosing-nodes-from-Z of game value $\alpha$, so in particular Black is losing in $G$. Then Black is losing in $p(G)$ and so $p(G)$ has a game value.
\end{lemma}
\begin{proof}
We exhibit a winning strategy for White.

Until White has won $G$, White follows the simple strategy of responding to {\tt b5} in any copy of Figure~\ref{fig:Z-component-unused} with {\tt cxb5}. Should Black follow up with {\tt c4}, White may play {\tt b6}, leaving Black with no more moves in that Z-component. This guarantees that White is never stalemated before White has won $G$, White gains an activated pawn chain to the root coming from each Z-component Black plays in, and that Black may only play up to two moves in each Z-component before he must play {\tt b5} in a new copy.

As Black must keep using new Z-components, White may continue this strategy until Black has lost $G$. Once this occurs, there exist two activated Z-components belonging to nodes $x$ and $y$ on the chessboard such that neither node is an initial segment of the other. Let $v$ be the lowest common ancestor node of $x$ and $y$ on the chessboard. As an illustrative example, certain nodes in Figure~\ref{fig:pawn-tree} are labeled as $x$, $y$, and $v$ respectively.

White can now win with the following strategy: White locally answers any Black pawn moves in Z-components north of $x$ or $y$ as before, in order to prevent any black pawn from being released and interfering with White's path to victory. Since there are only finitely many nodes on the board north of $x$ and $y$, Black will eventually run out of these moves and will be forced to move pawns at coordinates south of both $x$ and $y$ on the board (or at equal height), which White will now ignore. 

When not responding to moves in Z-components north of $x$ or $y$, White makes progress by slowly advancing up the pawn chains from nodes $x,y$ to $v$ in tandem, taking great care to only make moves in the currently lower of the two pawn chains. This precaution is because moving up a pawn chain releases a horde of black pawns in its wake, and White wants to prevent the progress in the two pawn chains from interfering with each other. After a finite number of moves, both pawn chains will have reached node $v$, where White can force the vertical release of a white pawn by Lemma~\ref{lem:normal_node}. This white pawn will now fight its way through the pristine pawn chains above it by using the crossing mechanism of Lemma~\ref{lem:crossing} at every chain until it hits the game-deciding pawn chain of the king chamber, forcing victory by Lemma~\ref{lem:bounded_king_chamber}.
\end{proof}

We are now ready to conclude the proof of Theorem~\ref{thm:game-values-pawns}:

\begin{proof}[Proof of Theorem~\ref{thm:game-values-pawns}]
    By Lemma~\ref{lem:choosing-nodes-game}, for every countable ordinal $\beta$ there is a game $G$ of \emph{choosing-nodes-from-Z} having countable game value at least $\beta$. By Lemma~\ref{lem:pawn-tree-sim2}, the associated chess position $p(G)$ has a game value for White, which by Lemma~\ref{lem:pawn-tree-sim1} is at least $\beta$. By a simple induction, a position with game value exactly $\beta$ can be reached after finite play from position $p(G)$, and so all countable game values arise among positions with only kings and pawns.
\end{proof}

Thus, the omega one of infinite chess with only kings and pawns is $\omega_1$ as well.

Though this construction is enough to answer the game value question, we note that it brings an inherent asymmetry between the players. In particular, it cannot simulate any game where White has meaningful choices. Therefore, we do not obtain a full analog of Theorem~\ref{thm:main} and are left with the following question:
\begin{question}
Can short-range infinite chess simulate arbitrary Gale-Stewart games with draws in the sense of Theorem~\ref{thm:main}?
\end{question}

\newpage

\bibliographystyle{plain}
\bibliography{Sources}

\end{document}